\documentclass{amsproc}

\usepackage{amsfonts}
\usepackage{amssymb}
\usepackage{mathrsfs}
\usepackage{xcolor,graphicx}
\usepackage{amsmath}
\usepackage{url}
\usepackage{verbatim}
\usepackage{hyperref}
\usepackage{xy}
\input xy
\xyoption{all}

\DeclareMathOperator{\Tot}{Tot}
\DeclareMathOperator{\holim}{\varprojlim}

\newcommand{\qcoh}{\mathrm{QCoh}}

\usepackage{amsthm}
\usepackage{cleveref/cleveref}

\newcommand{\dlin}{\Delta^{\mathrm{inj, \leq }n}}
\newcommand{\dl}{\Delta}

\newcommand{\dli}{\Delta^{\mathrm{inj}}}
\newcommand{\dln}{\Delta^{\mathrm{ \leq }n}}

\renewcommand{\hom}{\mathrm{Hom}}

\newtheorem{lm}{lm}
\newtheorem{lemma}[lm]{Lemma}
\newtheorem{corollary}[lm]{Corollary}

\newtheorem{theorem}[lm]{Theorem}

\newtheorem{proposition}[lm]{Proposition}

\begin{document}

\setlength{\parskip}{.7mm}
\renewcommand{\rightrightarrows}{\begin{smallmatrix} \to \\
\to \end{smallmatrix} }
\newcommand{\triplearrows}{\begin{smallmatrix} \to \\ \to \\ 
\to \end{smallmatrix} }

\newcommand{\sh}{\mathbf{Sh}}
 
\renewcommand{\ltimes}{\stackrel{\mathbb{L}}{\otimes}}
\newcommand{\psh}{\mathbf{PSh}}
\theoremstyle{definition}
\newtheorem{definition}[lm]{Definition}
\newcommand{\cd}{\mathrm{cd}}
\newcommand{\OO}{\widetilde{\mathcal{O}}}
\newcommand{\A}{\mathbb{A}}
\newcommand{\F}{\mathcal{F}}
\renewcommand{\P}{\mathbb{P}}
\newcommand{\bl}{\bullet}
\newcommand{\Tmf}{\mathrm{Tmf}}
\newcommand{\TMF}{\mathrm{TMF}}
\renewcommand{\ell}{\mathrm{Ell}}
\newcommand{\tmf}{\mathrm{tmf}}

\theoremstyle{definition}
\newtheorem{remark}[lm]{Remark}
\newcommand{\rng}{\mathbf{Ring}}
\newcommand{\gpd}{\mathbf{Gpd}}
\newcommand{\ei}{\mathbb{E}_1}
\renewcommand{\A}{\mathcal{A}_*}

\newcommand{\qcoha}{\qcoh^{\mathrm{ab}}}
\newtheorem{example}[lm]{Example}

\newcommand{\noteV}[1]{{\color{blue} \bf #1}} 

\title{Fibers of partial totalizations of a pointed cosimplicial space}
\date{\today}
\author{Akhil Mathew}
\address{University of California, Berkeley CA}
\email{amathew@math.berkeley.edu}

\author{Vesna Stojanoska}
\address{Max Planck Institute for Mathematics, Bonn, Germany}
\email{vstojanoska@mpim-bonn.mpg.de}
\thanks{
The first author is partially supported by the NSF Graduate Research
Fellowship under grant DGE-110640.
The second author is partially
supported by NSF grant DMS-1307390}

\maketitle

\newcommand{\e}[1]{\mathbf{E}_{#1}}

\begin{abstract}
Let $X^\bullet$ be a cosimplicial object in a pointed $\infty$-category. We show that the fiber of
$\mathrm{Tot}_m(X^\bullet) \to \mathrm{Tot}_n(X^\bullet)$ depends only on the
pointed cosimplicial object $\Omega^k X^\bullet$ and
is in particular a $k$-fold loop object, where $k  = 2n - m+2$. The
approach is explicit obstruction theory with
quasicategories. We also discuss generalizations to other types of homotopy
limits and colimits.  
\end{abstract}

\section{Introduction}
Let $X^\bullet$ be a pointed cosimplicial space, that is, $X^\bullet$ is a
functor $\Delta \to \mathcal{S}_*$ from the simplex category $\Delta$ to
the $\infty$-category  $\mathcal{S}_*$ of pointed spaces.  The
\emph{totalization} $\Tot X^\bullet$ is defined to be the homotopy limit of
$X^\bullet$. 

To understand $\mathrm{Tot} X^\bullet$, it is convenient to filter the category
$\Delta$. Let $\Delta^{\leq n}$ be the subcategory of $\Delta$ spanned by the
nonempty totally ordered sets of cardinality $\leq n+1$. 
Then $\mathrm{Tot} X^\bullet$
is the homotopy inverse limit of the
tower 
\begin{equation} \label{tottower}  \dots \to
\Tot_n(X^\bullet) \to
\Tot_{n-1}(X^\bullet) \to \dots \to \Tot_1(X^\bullet) \to
\Tot_0(X^\bullet) =
X^0, \end{equation}
where $\Tot_n(X^\bullet)$ is the homotopy limit of the diagram
$X^\bullet |_{\Delta^{\leq n}}$. 

The tower \eqref{tottower} is extremely useful. 
For example, it leads to a homotopy spectral sequence \cite[Ch. X]{BousfieldKan} for computing the homotopy groups of
the totalization (at the specified basepoint). 
For simplicity, we assume all fundamental groups 
\emph{abelian} and all $\pi_1$-actions on higher homotopy groups trivial. 
Then the $E_2$-term of this spectral sequence can be determined
purely algebraically, if one knows the homotopy groups of $X^\bullet$. Namely,
for each $t$, one obtains a cosimplicial abelian group $\pi_t X^\bullet$, and
$E_2^{s,t}$ is the $s$th cohomology group of the associated complex. 
This spectral sequence is a basic tool of algebraic topology. 

It is classical that the fibers of the maps $\mathrm{Tot}_m(X^\bullet) \to
\mathrm{Tot}_{m-1}(X^\bullet)$
are naturally $m$-fold loop spaces. To be specific, let $M^{m}X^\bullet$ be the $m$th matching space of the cosimplicial space $X^\bullet$, that is,
 \[M^m X^\bullet = \lim_{{ [m+1]} \twoheadrightarrow {[k]}\atop k \leq m} X^k.\]
 We have a natural map $X^m \to M^{m-1} X^\bullet$ with fiber denoted $N^m X^\bullet$, and the fiber of
 $\mathrm{Tot}_m(X^\bullet) \to \mathrm{Tot}_{m-1}(X^\bullet)$ can be
 identified with the $m$-fold loop space of $N^m X^\bullet$. 
 In particular, it depends only on $\Omega^m X^\bullet$, as $ M^{m-1}X^\bullet $ is defined as a limit and therefore commutes with $\Omega^m$.

The main result of this note is that one can obtain a similar picture for the fibers of
$\mathrm{Tot}_m(X^\bullet) \to \mathrm{Tot}_{n}(X^\bullet) $ for $n \leq m$,
in a restricted range, and in any pointed $\infty$-category, as follows.

\begin{theorem} 
\label{maintheorem}
Let $n \leq m \leq 2n+1$ and let $\mathcal{C}$ be any pointed $\infty$-category. 
There is a limit-preserving functor $F_{n,m}\colon \mathrm{Fun}(\Delta, \mathcal{C}) \to \mathcal{C}$
such that there exists,  for any $X^\bullet \in \mathrm{Fun}(\Delta, \mathcal{C})$, a functorial fiber sequence
in $\mathcal{C}$,
\[  F_{n,m}( \Omega^r X^\bullet) \simeq \Omega^r F_{n,m}(X^\bullet) \to \mathrm{Tot}_m(X^\bullet) \to
\mathrm{Tot}_n(X^\bullet)  \]
where $r = 2n-m + 2$. 
\end{theorem} 

In less formal terms, the fiber of $\mathrm{Tot}_m(X^\bullet) \to
\mathrm{Tot}_n(X^\bullet)$ depends
only on $\Omega^r X^\bullet$ as a cosimplicial object (and this fiber is in particular an
$r$-fold loop space). 
Our arguments are based on elementary obstruction-theoretic techniques, but in
the setting of the quasi-category model of $\infty$-categories. They in
particular apply to certain instances of homotopy (co)limits over finite posets
such as the poset of nontrivial subspaces of a finite-dimensional vector space. 

However, as a result, we do not obtain an explicit description of the functor
$F_{n,m}$: its existence is established by showing that a certain combinatorially defined
functor desuspends $r$ times. We show that such a desuspension \emph{exists} by
obstruction theory, but we do not write it down. 

In a similar vein, we obtain the following result.

\begin{theorem} 
\label{paracompact1}
Let $X$ be a pointed topological space covered by $n$ open sets $U_1$, $U_2$,
$\dots$, $U_n$ $\subset X$, each containing the basepoint. Suppose the intersection
of any collection of at most $r$ of the $\left\{U_i\right\}$ is
weakly contractible.
Then $X$ has the weak homotopy type of a $(2r - n  + 1)$-fold suspension. 
\end{theorem} 

Results of a similar flavor in Goodwillie calculus have been obtained by
Arone-Dwyer-Lesh \cite{ADL}. 

\subsection*{Acknowledgments} This result started out as a piece of 
the work in \cite{MS}
on the calculation of Picard groups of certain ring spectra. We are grateful to
Mike Hopkins for suggesting the original project and for several discussions,
and to Haynes Miller for helpful comments. 

\section{Deloopable subcategories}\label{sec:deloop}
We first consider this question more generally. Let $\mathcal{D}$ be an
$\infty$-category, and let $\mathcal{C} \subset \mathcal{D}$ be the inclusion
of a full subcategory. 
Given a functor $F\colon \mathcal{D} \to \mathcal{S}_*$, we can consider 
the fiber of the map of (homotopy) limits
\begin{equation} \label{limmap} \holim_{\mathcal{D}} F \to \holim_{\mathcal{C}} F ,\end{equation} and ask when this is a functor 
of some iterated loop space of $F$ (rather than simply of $F$). 

\begin{definition} \label{deloopabledef}
Let $\mathcal{C} \subset \mathcal{D}$ be as above. We say that the inclusion
$\mathcal{C} \to \mathcal{D}$  is \textbf{$k$-fold deloopable} if there is
an accessible, limit-preserving functor $\mathscr{G}\colon \mathrm{Fun}(\mathcal{D}, \mathcal{S}_*) \to
\mathcal{S}_*$ 
and a natural equivalence
\[ \mathscr{G}( \Omega^k F) \simeq \mathrm{fib}\left(  \holim_{\mathcal{D}} F \to
\holim_{\mathcal{C}} F\right), \quad F \in \mathrm{Fun}(\mathcal{D},
\mathcal{S}_*). \]
In other words, the fiber of \eqref{limmap} should functorially factor through
$\Omega^k$. 
\end{definition} 

\begin{example}
By the discussion in the introduction, $\Delta^{\leq n-1 } \subset \dln$ is $n$-fold deloopable. 
\end{example}

We begin by proving some elementary properties of $k$-fold deloopability. 
\begin{proposition} 
\label{kfolddeloopable}
Let $\mathcal{C} \subset \mathcal{D}$  be a fully faithful inclusion. Then the
following are equivalent: 
\begin{enumerate}
\item $\mathcal{C} \subset \mathcal{D}$ is $k$-fold deloopable.  
\item Let $\mathrm{Fun}_{\mathcal{C}}(\mathcal{D},
\mathcal{S}_*)$ be the full subcategory of $\mathrm{Fun}(\mathcal{D},
\mathcal{S}_*)$
consisting 
of those functors that take $\mathcal{C}$ to $\ast$. Then the functor
\[ \varprojlim_{\mathcal{D}}\colon  \mathrm{Fun}_{\mathcal{C}}(\mathcal{D},
\mathcal{S}_*) \to \mathcal{S}_* \]
factors through $\Omega^k$. In other words, there is an accessible,
limit-preserving functor $\mathscr{H}\colon \mathrm{Fun}_{\mathcal{C}}(\mathcal{D},
\mathcal{S}_*) \to \mathcal{S}_*$ together with a natural identification
$ \varprojlim_{\mathcal{D}} \simeq \mathscr{H} \circ \Omega^k $ of functors
$\mathrm{Fun}_{\mathcal{C}}(\mathcal{D},
\mathcal{S}_*) \to \mathcal{S}_*$. 
\end{enumerate}
\end{proposition} 
\begin{proof} 
Suppose $\mathcal{C} \subset \mathcal{D}$ is $k$-fold deloopable. 
Then, for $F \in \mathrm{Fun}_{\mathcal{C}}(\mathcal{D}, \mathcal{S}_*)$, 
$\varprojlim_{\mathcal{D}} F $ is naturally equivalent to the fiber of
$\varprojlim_{\mathcal{D}} F  \to \varprojlim_{\mathcal{C}} F$ since
$F|_{\mathcal{C}}$ is constant at $\ast$. But this factors through $\Omega^k$
by $k$-fold deloopability. 

Conversely, suppose the second condition satisfied. In this case, we note that
for any $F \in \mathrm{Fun}_{}(\mathcal{D}, \mathcal{S}_*)$, we have
\[ \varprojlim_{\mathcal{C}} F \simeq \varprojlim_{\mathcal{D}} \mathrm{Ran}_{\mathcal{C} \to
\mathcal{D}} F|_{\mathcal{C}}, \]
where $\mathrm{Ran}_{\mathcal{C} \to \mathcal{D}}$ denotes the right Kan
extension  from $\mathcal{C}$ to $\mathcal{D}$. We refer to \cite[\S
4.3]{highertopos} for a treatment of Kan extensions for $\infty$-categories. 
In particular, the fiber of \eqref{limmap} can be identified with 
\[ \varprojlim_{\mathcal{D}} \mathrm{fib}\left( F \to \mathrm{Ran}_{\mathcal{C}
\to \mathcal{D}} F|_{\mathcal{C}}\right),   \]
and the functor 
$\mathrm{fib}\left( F \to \mathrm{Ran}_{\mathcal{C}
\to \mathcal{D}} F|_{\mathcal{C}}\right) \in \mathrm{Fun}(\mathcal{D}, \mathcal{S}_*)$ has the property that it takes the
value $\ast$ on $\mathcal{C} \subset \mathcal{D}$, since $\mathcal{C} \subset
\mathcal{D}$ is a \emph{full} subcategory. Therefore, by the second
condition, this is a functor of $\Omega^k F$. 
\end{proof}

\begin{corollary} 
\label{preimagedeloop}
Let $i \colon \mathcal{C} \to \mathcal{D}$ be a fully faithful inclusion. Let
$p\colon \mathcal{D}' \to \mathcal{D}$ be a \emph{right cofinal} functor and let
$\mathcal{C}' \subset \mathcal{D}'$ be the subcategory of all objects that
map, under $p$, into the essential image of $i$. Then if $\mathcal{C}' \subset
\mathcal{D}'$ is $k$-fold deloopable, so is $\mathcal{C} \subset \mathcal{D}$.
\end{corollary}
\begin{proof} 

To see this, we use the criterion of \Cref{kfolddeloopable}. Consider any functor $F\colon \mathcal{D}
\to \mathcal{S}_*$ that maps everything in $\mathcal{C}$ to a weakly contractible
space. Then
$F \circ p$ maps everything in $\mathcal{C}' \subset \mathcal{D}'$ to $\ast$.
In particular, 
$ F \mapsto \lim_{\mathcal{D}} F = \lim_{\mathcal{D}'} F \circ p$ is a functor of $\Omega^k
F$. 
\end{proof}

We now give a basic criterion for $k$-fold
deloopability. Consider the $\infty$-category $\mathrm{Fun}( \mathcal{D},
\mathcal{S}_*)$ of
functors $\mathcal{D} \to \mathcal{S}_*$. 
If $F \in \mathrm{Fun}( \mathcal{D},
\mathcal{S}_*)$, then the homotopy limit $\holim_{\mathcal{D}}F $ is the
(pointed) mapping
space $\hom_{\mathrm{Fun}( \mathcal{D}, \mathcal{S}_*)}( \ast_+, F)$, where $\ast_+$ is the constant functor
with value $\ast_+ = S^0$. Analogously, the homotopy limit $\holim_{\mathcal{C}} F$ is given by 
\[ \holim_{\mathcal{C}} F \simeq 
\hom_{\mathrm{Fun}( \mathcal{C}, \mathcal{S}_*)}(\ast_+, F|_{\mathcal{C}}) . \]

Now let $\mathrm{Lan}_{\mathcal{C} \to \mathcal{D}}$ denote the left Kan
extension operation, from functors out of $\mathcal{C}$ to functors out of
$\mathcal{D}$. It follows by adjointness that
\[ \holim_{\mathcal{C}} F  \simeq 
\hom_{\mathrm{Fun}( \mathcal{D}, \mathcal{S}_*)}(\mathrm{Lan}_{\mathcal{C}
\to \mathcal{D}}(\ast_+), F) , \]
and the natural forgetful map 
\eqref{limmap}
is obtained from the counit map
\[ \mathrm{Lan}_{\mathcal{C} \to \mathcal{D}}(\ast_+) \to \ast_+,  \]
in $\mathrm{Fun}(\mathcal{D}, \mathcal{S}_*)$. 
Therefore, it follows that the fiber of \eqref{limmap} (over the given basepoint, that is) is given by 
the functor
\[ F \mapsto \hom_{\mathrm{Fun}( \mathcal{D}, \mathcal{S}_*)}( T, F),  \]
where the functor $T = T_{\mathcal{C}, \mathcal{D}}\colon \mathcal{D} \to \mathcal{S}_*$ is the homotopy pushout
\begin{equation} \label{T}\xymatrix{
\mathrm{Lan}_{\mathcal{C} \to \mathcal{D}}(\ast_+) \ar[d]  \ar[r] &  \ast \ar[d]  \\
\ast_+ \ar[r] &  T
}.\end{equation}

When the \emph{functor} $T\colon \mathcal{D} \to \mathcal{S}_*$ can be desuspended
$k$ times, so that $T \simeq \Sigma^k T'$ for some
$T' \in \mathrm{Fun}( \mathcal{D}, \mathcal{S}_*)$, then the fiber of
\eqref{limmap} is given by $\Omega^k \hom_{\mathrm{Fun}(\mathcal{D},
\mathcal{S}_*)}(T', F)$. 
The main strategy of this paper is to use a direct obstruction-theoretic
argument to make such desuspensions. 

Combining these ideas, we get:
\begin{proposition} 
Let $\mathcal{C} \subset \mathcal{D}$ be an inclusion of a full subcategory.
Then the following are equivalent: 
\begin{enumerate}
\item $\mathcal{C}  \subset \mathcal{D}$ is $k$-fold deloopable. 
\item The functor $T \colon\mathcal{D} \to \mathcal{S}_*$ defined by \eqref{T} can
be functorially desuspended $k$ times.
\end{enumerate}
\end{proposition} 
\begin{proof} 
The fiber of \eqref{limmap}, as a functor $\mathrm{Fun}(\mathcal{D},
\mathcal{S}_*) \to \mathcal{S}_*$, is corepresentable by $T$. Thus, by
Yoneda's lemma, to deloop
the functor \eqref{limmap} $k$ times by a \emph{corepresentable} functor is
equivalent to desuspending the corepresenting object $T$ $k$ times. Any functor
$\mathscr{G}$ as in \Cref{deloopabledef} is corepresentable by the adjoint
functor theorem \cite[Proposition 5.5.2.7]{highertopos}. 
\end{proof}

\newcommand{\fun}{\mathrm{Fun}}
There is no reason to restrict to functors taking values in $\mathcal{S}_*$: we
could take any pointed $\infty$-category with enough limits. For example, we could take
$\mathcal{S}_*^{\mathrm{op}}$ so as to get a result about $\mathcal{C}$-valued
and $\mathcal{D}$-valued homotopy \emph{colimits} of spaces. 
Indeed, let $\mathcal{A}$ be any pointed $\infty$-category with limits. 
In this case, given a functor $f\colon \mathcal{D} \to \mathcal{S}_*$ and a functor
$F\colon \mathcal{D} \to \mathcal{A}$
, we can form
an internal mapping object $\hom_{\fun(\mathcal{D}, \mathcal{A})}(f, F) \in
\mathcal{A}$. 
This construction is determined by the properties that it should send homotopy
colimits in $f$ to homotopy limits and that if $f$ is $\hom(d,
\cdot)_+$ for $d \in \mathcal{D}$, then 
$\hom_{\fun(\mathcal{D}, \mathcal{A})}(f, F) \in
\mathcal{A}$ is given by $F(d)$. If $f$ is $\ast_+$, then one gets the homotopy
limit of $F$. 
In particular, we find that 
the fiber of the map 
\( \varprojlim_{\mathcal{D}} F \to  \varprojlim_{\mathcal{C}} F \) can be
identified with $\hom_{\fun(\mathcal{D}, \mathcal{A})}( T, F)$ as before. In
particular, if $T$ is $k$-fold desuspendable, then this fiber can be obtained as a
$k$-fold loop space. 
We state this formally.

\begin{corollary} 
Let $\mathcal{A}$ be any pointed $\infty$-category with limits. Then if $\mathcal{C}
\subset \mathcal{D}$ is a $k$-fold deloopable inclusion, the functor
$\fun(\mathcal{D}, \mathcal{A}) \to \mathcal{A}$ sending $F \in 
\fun(\mathcal{D}, \mathcal{A})$ to the 
fiber of
$\varprojlim_{\mathcal{D}} F \to \varprojlim_{\mathcal{C}}F$ can be expressed as $\mathscr{G} \circ \Omega^k$ for a
limit-preserving functor $\mathscr{G}\colon \fun(\mathcal{D}, \mathcal{A}) \to
\mathcal{A}$. 
\end{corollary}

We recall, for future reference, the explicit construction of the left Kan
extension, in the case $\mathcal{C}, \mathcal{D}$ are ordinary
categories. 
Given ordinary categories $\mathcal{C} \subset \mathcal{D}$, the left Kan extension 
$\mathrm{Lan}_{\mathcal{C} \to \mathcal{D}}(\ast) \colon \mathcal{D}
\to \mathcal{S}$
of the functor
$\mathcal{C} \to \mathcal{S}$ constant at $\ast$ (where we work with
\emph{unpointed} spaces) given by the functor to the category of simplicial sets
\[ \mathcal{D} \to \mathcal{S}, \quad d \mapsto N( \mathcal{C}_{/d}),  \]
where $\mathcal{C}_{/d}$ is the fiber product $\mathcal{C} \times_{\mathcal{D}}
\mathcal{D}_{/d}$ for $\mathcal{D}_{/d}$ the overcategory.
Note in particular that $T(c)$ is weakly contractible, for any $c \in \mathcal{C}$. 

\begin{remark}
Observe that \eqref{T} exhibits $T$ as an \emph{unreduced suspension} of
the functor $\mathrm{Lan}_{\mathcal{C} \to \mathcal{D}}(\ast) \colon
\mathcal{D} \to \mathcal{S}$. 
It follows that if 
$\mathrm{Lan}_{\mathcal{C} \to \mathcal{D}}(d) $ is a nonempty space for each
$d \in \mathcal{D}$, then each $T(d) \in \mathcal{S}_*$  can be written
as a suspension in $\mathcal{S}_*$.
In our primary
example of interest, $T(d)$ is a wedge of spheres for each $d$. Thus, even the
individual maps  $T(d_1) \to T(d_2)$ will be suspension maps. As a result, the
composite functor $\mathcal{D} \stackrel{T}{\to} \mathcal{S}_* \to
\mathrm{Ho}(\mathcal{S}_*)$ factors through $\Sigma$.
However, that alone does not mean
we can desuspend the functor $T$. 

The distinction between diagrams in the \emph{homotopy category} and in the
actual $\infty$-category $\mathcal{S}_*$ (e.g., presented by a model category
such as that of simplicial sets) is central to this paper. These ideas are
classical and have been studied in detail by 
Cooke \cite{cooke} for group actions and more generally by 
Dwyer and Kan in \cite{DK1, DK2, DK3}. We will not need the
sophisticated cohomological machinery that these authors develop, though. 
\end{remark}
\section{Obstruction theory}
Let $\mathcal{C} \subset \mathcal{C}'$ be an inclusion. 
In order to understand whether the functor $T_{\mathcal{C}, \mathcal{C}'}\colon
\mathcal{C}' \to \mathcal{S}_*$ introduced in the previous section can be
desuspended a certain number of times, we will need to introduce a
small amount of obstruction theory for $\infty$-categories. While such
obstruction theories are (in
various forms) classical, we will use very little. It will be convenient to 
use the \emph{quasi-category} model of $\infty$-categories.

Let $K$ be a simplicial set. Let $\mathcal{C}, \mathcal{D}$ be quasicategories
and $F\colon \mathcal{C} \to \mathcal{D}$ a functor (not necessarily an
inclusion). In our setting, it will be the
suspension functor $\Sigma\colon \mathcal{S}_* \to \mathcal{S}_*$ restricted to a
subcategory of $\mathcal{S}_*$, as we are trying to desuspend the functor $T$ from the
previous section. 
Our goal is to contemplate lifting diagrams of the form
\begin{equation} \label{lift} \xymatrix{
& \mathcal{C} \ar[d]^F \\
K \ar[r]^f \ar@{-->}[ru] &  \mathcal{D}
}\end{equation}
Such a lifting diagram is understood to take place in the \emph{homotopy}
category of quasicategories (e.g., of the Joyal model structure on
simplicial sets).  

\begin{proposition} 
\label{lifting}
Suppose $F \colon \mathcal{C} \to \mathcal{D}$ 
has the property that the essential image of $F$ contains that of $f$. 
Suppose $\dim K \leq m$ and suppose that 
for each $c_1, c_2 \in \mathcal{C}$, the map 
\[ F\colon \hom_{\mathcal{C}}(c_1, c_2) \to \hom_{\mathcal{D}}(Fc_1, Fc_2)  \]
is $(m-1)$-connected. Then a lift in \eqref{lift} always exists. 
\end{proposition} 
\begin{definition} 
\label{mconnecteddef}
For $m \geq 0$, we will say that a functor $F\colon \mathcal{C} \to \mathcal{D}$ is
\textbf{$m$-connected} if $F$ is essentially surjective and 
for each $c_1, c_2 \in\mathcal{C}$, the map
$F\colon \hom_{\mathcal{C}}(c_1, c_2) \to \hom_{\mathcal{D}}(Fc_1, Fc_2)$ is
$(m-1)$-connected.
\end{definition} 

\begin{proof} 
In view of the equivalence between quasicategories and simplicial categories
\cite[Theorem 2.2.5.1]{highertopos}, 
we can assume without loss of generality that the quasicategories
$\mathcal{C}, \mathcal{D}$ are
the homotopy coherent nerves of fibrant simplicial categories $\mathscr{C},
\mathscr{D}$ (using the model structure of Bergner \cite{bergner}). 
Moreover, we can assume that $F \colon \mathcal{C} \to \mathcal{D}$ is the
homotopy coherent nerve of a simplicially enriched functor $F \colon
\mathscr{C} \to \mathscr{D}$ which in turn is a fibration in the Bergner model
structure. In this case, we claim that any lifting problem 
\eqref{lift} can be solved in the category of simplicial sets itself.
Inducting over the skeleta, it now suffices to show that in any diagram
of simplicial sets, for $k \leq m$,
\[ \xymatrix{
\partial \Delta^k \ar[d] \ar[r] & \mathcal{C} \ar[d] \\
\Delta^k \ar[r] \ar@{-->}[ru] &  \mathcal{D}
},\]
there exists a lift.

In this case, we need to solve the
equivalent lifting problem (in simplicial categories) for 
\[ \xymatrix{
\mathfrak{C}[\partial \Delta^k] \ar[d] \ar[r] & \mathscr{C} \ar[d]  \\
\mathfrak{C}[\Delta^k] \ar[r] \ar@{-->}[ru]&  \mathscr{D}
},\]
where we have used the notation as in \cite[Def. 1.1.5.1]{highertopos}. The
equivalence 
follows from the definition of the homotopy coherent nerve as the right adjoint
to $\mathfrak{C}$.

However, we note that the simplicial functor $\mathfrak{C}[\partial \Delta^k] \to
\mathfrak{C}[\Delta^k]$ is obtained via a pushout along a cofibration in the
Bergner model structure. For a simplicial set 
$L$, let $\mathfrak{D}[L]$ denote the simplicial  category 
with two objects $a,b$ with $\hom(a,a) = \hom(b,b) = \Delta^0, \hom(a,b) = L,
\hom(b,a) = \emptyset$.  
Then there is a pushout diagram in simplicial categories
\[ \xymatrix{
\mathfrak{D}[ X_k] \ar[d] \ar[r] & \mathfrak{C}[\partial \Delta^k] \ar[d] \\
\mathfrak{D}[Y_k] \ar[r] &  \mathfrak{C}[ \Delta^k]
},\]
where $X_k$ has the homotopy type of $S^{k-2}$ and $Y_k$ the homotopy type of a
point; moreover, $X_k \to Y_k$ is a cofibration. For this, compare the discussion on \cite[p. 90]{highertopos} in the proof of
\cite[Theorem 2.2.5.1]{highertopos}.
Therefore, our lifting problem is equivalent to that given by the diagram
of simplicial categories
\begin{equation} \label{liftsimplcat} \xymatrix{
\mathfrak{D}[X_k] \ar[d] \ar[r] & \mathscr{C} \ar[d] \\
\mathfrak{D}[Y_k] \ar@{-->}[ru] \ar[r] &  \mathscr{D}
}.\end{equation}

Unwinding the definitions, it follows that our lifting problem 
\eqref{liftsimplcat}
is equivalent to 
a lifting problem in simplicial sets \[  \xymatrix{
X_k \ar[d]  \ar[r] &  \hom_{\mathscr{C}}(a,b)  \simeq
\hom_{\mathcal{C}}(a,b)\ar[d] \\
Y_k \ar[r] \ar@{-->}[ru] &  \hom_{\mathscr{D}}(a,b) \simeq
\hom_{\mathcal{D}}(a,b)
},\]
and we claim that we can solve this by the connectivity hypotheses.

In fact, 
suppose $U \to U'$ is a fibration of Kan complexes which is $(s-1)$-connected. 
Suppose $r \leq s-1$.
Then in any diagram of simplicial sets \[ \xymatrix{
X \ar[d] \ar[r] &  U \ar[d] \\
Y \ar[r] \ar@{-->}[ru] &  U'
},\]
with $X \to Y$ a cofibration such that $X$ has the homotopy type of $S^{r-1}$
and $Y$ has the homotopy type of $\ast$,
there exists a lift by obstruction theory. 
In our case, note that $\hom_{\mathscr{C}}(a, b) \to \hom_{\mathscr{D}}(a,b)$ is a Kan
fibration because $\mathscr{C} \to \mathscr{D} $ is a fibration of simplicial
categories.
\end{proof} 

\begin{remark} 
The simplicial sets that one encounters in the above argument will not all be
quasi-categories, even if $K$ is a quasi-category. It is a convenient feature of
the simplicial model that
one can still make ``cell'' decompositions in this way. 
\end{remark} 

In the rest of this section, we will note two simple tools for working with
the above ideas which we will use in the sequel. 

\begin{proposition} \label{dimcat} \label{dimension} Fix an integer $m$. 
Let $\mathcal{C}$ be an ordinary category with the property that whenever
$c_0, \dots, c_{m+1} $ are given objects in $ \mathcal{C}$ and  $c_i \to c_{i+1}$, for all $0\leq i \leq m$, are given morphisms, then (at least) one
of the morphisms is the identity. Then the nerve of $\mathcal{C}$ has dimension $ \leq
m$ as a simplicial set. 
\end{proposition} 
For instance, if $\mathcal{C}$ is a poset, then it follows that the dimension
is one less than the maximum size of a \emph{chain.}
\begin{proof} 
In fact, every nondegenerate simplex in the nerve of $\mathcal{C}$ consists of
such a tuple of morphisms as above with none the identity, so that every nondegenerate simplex must
have dimension $\leq m$. 
\end{proof}

Finally, we include an auxiliary result that enables us to simplify the 
process of finding
liftings by leaving out certain parts of the category. 
\begin{proposition} 
\label{subposet}
Let $P$ be a poset and $P' \subset P$ a subposet such that no element of $P'$
is less than any element of $P \setminus P'$.
Let $\mathrm{Fun}'( N(P), \mathcal{S}_*)$ denote the full subcategory of
$\mathrm{Fun}(N(P), \mathcal{S}_*)$ spanned by those functors that send $P
\setminus P' $ to $\ast$. Then the restriction functor
\[  \mathrm{Fun}'( N(P), \mathcal{S}_*) \to \mathrm{Fun}( N(P'), \mathcal{S}_*) \]
is an equivalence of $\infty$-categories. 
\end{proposition} 
\begin{proof} 
Observe that as a simplicial set, $N(P)$ is built from $N(P')$ by iteratively
attaching $n$-simplices $\Delta^n$ along the boundary $\partial \Delta^n$, where the initial vertex is necessarily in
$P \setminus P'$.
In other words, there exists a finite sequence of simplicial sets
\[ N(P') = K_0 \to K_1 \to \dots \to K_r = N(P),  \]
where
each map $K_i \to K_{i+1}$ fits into a pushout square
\[ \xymatrix{
\partial \Delta^{t(i)} \ar[d] \ar[r] &  \Delta^{t(i)} \ar[d]^{\phi_i} \\
K_i \ar[r] &  K_{i+1}
}.\]
Moreover, the vertex $0$ of each $\partial \Delta^{t(i)}$ is mapped to an
element of $P \setminus P'$. 
Note that $K_i \subset N(P)$, and we let $V_i \subset (K_i)_0$ denote the
collection of vertices in $K_i$ that belong to $ P \setminus P'$.
Note that $V_i \subset V_{i+1}$ with equality unless $t(i) = 0$, in which case
$V_{i+1}$ is the union of $V_i$ and the image of 
the map $\phi_i \colon \Delta^0 \to K_{i+1}$.

For each $i$, we let $\mathrm{Fun}'( K_i, \mathcal{S}_*)$ denote the
subcategory of $\mathrm{Fun}(K_i, \mathcal{S}_*)$ spanned by those functors
which carry every vertex in $V_i$ to $\ast$. 
We then 
have 
$\mathrm{Fun}'( K_0, \mathcal{S}_*) = \mathrm{Fun}( N(P'), \mathcal{S}_*)$ and  
$\mathrm{Fun}'(K_r, \mathcal{S}_*) = \mathrm{Fun}'( N(P), \mathcal{S}_*)$. 
For each $i$,
we have a 
pullback diagram of $\infty$-categories
\[ 
\xymatrix{
\mathrm{Fun}'(K_{i+1}, \mathcal{S}_*) \ar[d] \ar[r] & \mathrm{Fun}'( K_i,
\mathcal{S}_*) \ar[d] \\
\mathrm{Fun}_{0 \to \ast}( \Delta^{t(i)}, \mathcal{S}_*) \ar[r] & 
\mathrm{Fun}_{0 \to \ast}( \partial \Delta^{t(i)}, \mathcal{S}_*) 
},
\]
where the $0 \to \ast$ subscript indicates that the initial vertex is sent to
$\ast$.

By induction, it thus suffices to show that the forgetful functor
\[ \mathrm{Fun}_{0 \to \ast}(\Delta^n, \mathcal{S}_*) \to
\mathrm{Fun}_{0 \to \ast}(\partial
\Delta^n, \mathcal{S}_*)   \]
 is an equivalence of $\infty$-categories. This follows from the
next lemma. 
\end{proof}

\begin{lemma} 
Let $\mathcal{C}$ be a quasi-category containing an initial
object. Let $\mathrm{Fun}_{0 \to \ast}( \Delta^n, \mathcal{C})$ be the full
subcategory of the quasi-category $\mathrm{Fun}( \Delta^n, \mathcal{C})$ spanned by those functors
that send $0 \in \Delta^n$ to an initial object, and define 
$\mathrm{Fun}_{0 \to \ast}( \partial \Delta^n, \mathcal{C})$ similarly. Then 
the restriction map
$$\mathrm{Fun}_{0 \to \ast}( \Delta^n, \mathcal{C})\to \mathrm{Fun}_{0 \to \ast}( \partial \Delta^n, \mathcal{C})$$
is a trivial fibration of simplicial sets. 
\end{lemma} 
\begin{proof} 
For $n = 0$, this assertion states that the subcategory of $\mathcal{C}$
spanned by the initial objects is a contractible Kan complex, which is proved
in \cite[Proposition 1.2.12.9]{highertopos}. We thus assume $n \geq 1$. 
We need to show that any lifting problem
\[ \xymatrix{
\partial \Delta^m \ar[d] \ar[r] &  \mathrm{Fun}_{0 \to \ast}( \Delta^n,
\mathcal{C}) \ar[d]  \\
\Delta^m \ar[r] \ar@{-->}[ru] &  \mathrm{Fun}_{0 \to \ast}( \partial \Delta^n,
\mathcal{C})
},\]
can be solved. This lifting problem is equivalent to one of the form
\begin{equation}
\label{secondlift}
\xymatrix{
\partial \Delta^n \times \Delta^m \sqcup_{\partial \Delta^n \times \partial
\Delta^m} \Delta^n \times \partial \Delta^m \ar[d] \ar[r] &  \mathcal{C} \\
\Delta^n \times \Delta^m \ar@{-->}[ru]
}
.\end{equation}
The given 
map
$\partial \Delta^n \times \Delta^m \sqcup_{\partial \Delta^n \times \partial
\Delta^m} \Delta^n \times \partial \Delta^m \to \mathcal{C}$ carries every vertex $(0, i)$ (for
$i \in \left\{0, 1, \dots, m\right\}$) to an initial object of $\mathcal{C}$. 
Moreover, the inclusion of simplicial sets
$\partial \Delta^n \times \Delta^m \sqcup_{\partial \Delta^n \times \partial
\Delta^m} \Delta^n \times \partial \Delta^m \to \Delta^n \times \Delta^m$ is
built up by successively attaching simplices $\Delta^p$, along the boundary
$\partial \Delta^p$, whose initial
vertex is $(0, 0)$. 
In particular, to find the lift in 
\eqref{secondlift}, it suffices to contemplate a problem of the form
\[ \xymatrix{
\partial \Delta^p \ar[d] \ar[r] & \mathcal{C} \\
\Delta^p  \ar@{-->}[ru]
},\]
where the initial vertex of $\partial \Delta^p$ maps to an initial object of
$\mathcal{C}$. But by the definition of an initial object of a quasi-category
\cite[Definition 1.2.12.3]{highertopos}, such lifting problems can always be
solved. 
\end{proof}

\section{Specialization to totalizations}

In this section, we will prove our main result on $k$-fold deloopability for
totalizations (which is a restatement of \Cref{maintheorem}). 

\begin{theorem} 
\label{kdeloopable}
Let $1 \leq n \leq m \leq 2n+1$.  Then $\Delta^{\leq
n} \subset \Delta^{\leq m}$ is $(2n-m+2)$-fold deloopable. \end{theorem} 

In order to prove \Cref{kdeloopable}, we will follow the outline already laid
out in the previous sections. Namely, we will write down the functor $T$ introduced in \Cref{sec:deloop} explicitly;
it will be pointwise a wedge of spheres. Then we will show, using the obstruction
theory, that it can be desuspended the right number of times.

Our first problem is, however, that that the nerve $N( \dln)$ is not a
finite-dimensional simplicial set, and it will be convenient to replace it by
one.

\begin{definition} 
As before, let $\dl$ be the simplex category, of finite nonempty totally 
ordered sets and order-preserving maps between them. We let $\dli \subset \dl$ be the
subcategory with the same objects but \emph{injective} order-preserving maps. 
We let $\dlin$ and $\dln$ denote the analogs where we only consider 
totally ordered sets of cardinality $\leq n+1$. 
\end{definition} 

It is classical that the inclusion $\dli \subset \dl$
is (homotopy) right cofinal; see for instance \cite[Example 8.5.12]{riehl}. 
The advantage of the category $\dli$ is that the nerve of a skeleton of $\dlin$ is
$n$-dimensional, as any composable sequence $x_0 \to x_1 \to \dots \to x_{n+1}$ of arrows in a skeleton
of $\dlin$ contains
an identity. (Compare \Cref{dimcat}.)
However, the inclusion $\dlin \subset \dln$ is no longer homotopy final: 
$\dln$ is weakly contractible (as it has a terminal object) while $\dlin$
fails to be weakly contractible (for instance for $n = 1$). 
Recall that any right cofinal map of quasi-categories is in particular a weak
homotopy equivalence of underlying simplicial sets 
\cite[Proposition 4.1.1.3]{highertopos}. 

Let $\dlin_{/[n]}$ be the overcategory; it is equivalent to the partially ordered set of all
nonempty subsets of $[n] = \left\{0, 1, \dots, n\right\}$. We have a natural
functor $\dlin_{/[n]} \to \dlin_{} \to \dln$. 
Now we use:

\begin{proposition}[{}]
\label{dln}
The composite $\dlin_{/[n]} \to \dln$ is right cofinal. 
\end{proposition} 
References for \Cref{dln}, which is folklore but technical,  include
\cite[Proposition 18.7]{duggercolim} and 
\cite[Lemma 1.2.4.17]{higheralg}. 

We will thus work with the $\dlin_{/[n]}$. 
Using \Cref{preimagedeloop}, it suffices to prove an analog of
$k$-fold deloopability here; the advantage is that, by \Cref{dimension}, the nerve
has dimension $n$ (up to taking a skeleton). 

First, we will change notation. 
Let $S$ be a finite, nonempty set (e.g., $[n]$). We let $\mathcal{P}(S)$
denote the partially ordered set of all nonempty subsets of $S$, which by
abuse of notation we will identify with its nerve. 
Let $\mathcal{P}_{\leq r}(S) \subset \mathcal{P}(S)$ be the partially ordered
subset of all nonempty subsets of $S$ of cardinality $\leq r$. 
It follows by \Cref{preimagedeloop} and \Cref{dln} that our main result,
\Cref{kdeloopable}, is a corollary (with $r = n+1, S = [m]$, $|S| = m+1$) of:

\begin{proposition} 
\label{powsetdeloop}
$\mathcal{P}_{\leq r}(S) \subset \mathcal{P}(S)$ is $(2r -
|S|+1)$-fold deloopable. 
\end{proposition} 

\begin{proof}
The left Kan extension
\( \mathrm{Lan}_{\mathcal{P}_{\leq r}(S) \to \mathcal{P}(S)}(\ast)  \)
can be described explicitly by the following formula: it assigns to an
arbitrary nonempty subset 
$U \subset S$ the nerve of the partially ordered set $\mathcal{P}_{\leq r}(U)$ of all nonempty subsets of
$U$ of cardinality $\leq r$. 
In general, for any finite set $U$, the nerve of $\mathcal{P}_{\leq r}(U)$ is isomorphic to the barycentric subdivision of the
simplicial \emph{complex} obtained as the union of all the $(r-1)$-simplices of an
$(|U|-1)$-simplex. This has the homotopy type of a wedge of $(r-1)$-spheres,
as one sees by comparing the simplicial homology to that of the standard
simplex. 
It follows that the functor
$$T = T_{\mathcal{P}_{\leq r}(S) \to \mathcal{P}(S)}\colon \mathcal{P}(S) \to
\mathcal{S}_*$$ is, pointwise,
a wedge of $r$-spheres: it is the unreduced suspension of the left Kan
extension $\mathrm{Lan}_{\mathcal{P}_{\leq r}(S) \to \mathcal{P}(S)}(\ast)$. 
Moreover, the functor 
$T$ is weakly contractible when evaluated on any subset of $S$ of cardinality $\leq r$. 

Our goal is to functorially desuspend $T$. 
Consider the subset $\mathcal{P}_{>r}(S)$ of the poset $\mathcal{P}(S)$ where we consider
subsets of $S$ of cardinality $> r$. 
By \Cref{subposet}, it is equivalent to desuspend $T|_{\mathcal{P}_{>r}(S)}$. 
By \Cref{dimension}, 
the nerve of $\mathcal{P}_{>r}(S)$ has dimension $|S| - r-1$. 

Let $d> 0$. 
We are thus contemplating 
the lifting problem
\[ \xymatrix{
& \mathcal{S}_* \ar[d]^{\Sigma^d} \\
\mathcal{P}_{>r}(S) \ar@{-->}[ru]\ar[r]^T & \mathcal{S}_*
}.\]

We do not need to work with all of $\mathcal{S}_*$ for this. For each $p$, 
consider the $\infty$-category $\mathcal{C}(p) \subset \mathcal{S}_*$ which is
the full subcategory spanned by the
finite wedges of
$p$-spheres. 
One has a functor
\[ \Sigma\colon \mathcal{C}(p) \to \mathcal{C}(p+1),  \]
and our problem is to find a lift
\[ \xymatrix{
& \mathcal{C}(r-d)\ar[d]^{\Sigma^d} \\
\mathcal{P}_{>r}(S) \ar@{-->}[ru]\ar[r]^T &\mathcal{C}(r)
}.\]

Observe that, for any $p$, the maps
given by $\Sigma$,
\[ \hom_*(S^p \vee \dots \vee S^p, S^p \vee \dots \vee S^p) \to
\hom_*(S^{p+1} \vee \dots \vee S^{p+1}, S^{p+1} \vee \dots
\vee S^{p+1}),  \]
are $(p-1)$-connected by the Freudenthal suspension theorem. 
Therefore, the functor $\Sigma^d\colon \mathcal{C}(r-d) \to \mathcal{C}(r)$ above is $ (r-d)$-connected in the sense of
\Cref{mconnecteddef}.
By
\Cref{lifting}, we can find a lift if 
\[ \dim N(\mathcal{P}_{>r}(S)) = |S| - r-1 \leq  r-d,  \]
or equivalently if 
\[ d \leq 2r - |S|+1,  \]
as desired. 
In particular, we can desuspend $T$ by $2r - |S|+1$ times, which proves that the
inclusion in question is $(2r - |S| + 1)$-fold deloopable. 
\end{proof}

\begin{remark} 
By \Cref{kdeloopable}, we find that if $X_\bullet$ is any pointed simplicial space, then the
cofiber of the geometric realization of the $m$-truncation by the geometric
realization of the $n$-truncation is a $k$-fold suspension, for $k  =2n - m +
2$. This cannot be improved, since (using the standard simplicial
model given by the bar construction) $\mathbb{RP}^m/\mathbb{RP}^n$ cannot be a
suspension if $m \geq 2n+2$ and nontrivial cup products exist. In particular, \Cref{kdeloopable} is best possible (at least for one-fold
deloopings). 
\end{remark} 

It will actually be convenient to have a slightly stronger result. 
\begin{corollary} \label{strongdeloop}
Let $1 \leq n \leq m \leq 2n+1$ and let $k = 2n-m+2$ and let $\mathcal{C}$ be
any pointed $\infty$-category with limits. 
Then for $X^\bullet \in \fun( \Delta, \mathcal{C})$, the sequence of objects
\[ \mathrm{fib}\left(\mathrm{Tot}_m(X^\bullet) \to
\mathrm{Tot}_n(X^\bullet)\right) \to 
 \mathrm{fib}\left(\mathrm{Tot}_{m-1}(X^\bullet) \to
 \mathrm{Tot}_n(X^\bullet)\right)  
 \to \dots \to \ast
\]
is a functor of $\Omega^k X^\bullet$. 
\end{corollary} 

\newcommand{\ran}{\mathrm{Ran}}
Each of the objects in question is naturally a functor of $\Omega^k X^\bullet$
by \Cref{kdeloopable}; the additional claim here is that the \emph{maps} can be made into $k$-fold
loop maps. 
\begin{proof} 
We consider the sequence of objects in $\fun( \Delta, \mathcal{C})$,
\[ \ran_{\Delta^{\leq m} \to \Delta}(X^\bullet) \to
\ran_{\Delta^{\leq m-1} \to \Delta}(X^\bullet)
\dots \to 
\ran_{\Delta^{\leq n} \to \Delta}(X^\bullet).
\]
Form the fibers of all of these over the last one to get a sequence of pointed
cosimplicial spaces
\[ \left\{\mathrm{fib}\left(\ran_{\Delta^{\leq i} \to \Delta}(X^\bullet) \to
\ran_{\Delta^{\leq n} \to \Delta}(X^\bullet)\right) \right\}_{n \leq i \leq m} 
.\]
This is a sequence of pointed cosimplicial spaces, each of which restricts to $\ast$ on
$\Delta^{\leq n}$. When we take $\mathrm{Tot}_m$ of this sequence, the result
is therefore a functor of $\Omega^k X^\bullet$ by \Cref{kfolddeloopable} and
\Cref{kdeloopable}. 
But when we take $\mathrm{Tot}_m$ of this sequence, we obtain precisely the
sequence of pointed spaces in the statement of the corollary. 
\end{proof}

\section{Applications and extensions}

\newcommand{\sub}{\mathrm{Sub}}

\subsection{Open covers}
We start with an application to the most classical example of a homotopy
colimit, which was stated in the introduction (\Cref{paracompact1}).  

\begin{theorem} \label{desuspension}
Let $X$ be a pointed topological space covered by $n$ open sets $U_1$, $U_2$,
$\dots$, $U_n$ $\subset X$, each containing the basepoint. Suppose the intersection
of any collection of at most $r$ of the $\left\{U_i\right\}$ is
weakly contractible.
Then $X$ has the weak homotopy type of a $(2r - n  + 1)$-fold suspension. 
\end{theorem} 
\begin{proof} 
For a space $X$ covered by open sets $U_1, \dots, U_n$ as above, there is a weak equivalence \cite{DuggerIsaksen}
between $X$ and the homotopy colimit over the poset $\mathcal{P}(\left\{1, 2, \dots, n\right\})$ 
of the functor that sends a nonempty subset $S \subset \left\{1, 2, \dots, n\right\}$ to
$\bigcap_{i \in S} U_i$.
This is a \emph{contravariant} functor from the poset
to $\mathcal{S}_*$, and by assumption, it takes weakly contractible values on the
subsets of $\left\{1, 2, \dots, n\right\}$ of cardinality $\leq r$. Using
\Cref{powsetdeloop}, we conclude that the homotopy colimit has the weak
homotopy type of a
$(2r - n + 1)$-fold suspension.
\end{proof} 

Suppose we are in the \emph{unpointed} setting, and we have an open cover
$U_1, \dots, U_n$ of a space $X$ such that all the intersections of
the $\left\{U_i\right\}$ are either
{empty} or {weakly contractible.} In this case, $X$ has the weak homotopy type of
the \emph{nerve}
of the cover, which has dimension $\leq n-1$ as a simplicial set. 
If the dimension is $n-1$, then $U_1 \cap \dots \cap U_n \neq \emptyset$ and is
thus weakly contractible, so that $X$ is itself weakly contractible. Otherwise, the dimension
is $\leq n-2$. We consider only this case. 

Suppose the intersection of any $r$ of the $\left\{U_i\right\}$ is nonempty. 
Then, $X$ is $(r-2)$-connected. 
But any complex with cells in dimensions $[s, t]$ can be desuspended (by the
Freudenthal suspension theorem) if $t \leq 2s - 1$. 
Here, we find that if $n-2 \leq 2(r-1) - 1$, or $n \leq 2r-1$, then, up to
weak equivalence, $X$ can be
desuspended (after choosing a basepoint). This is a version of the conclusion of \Cref{desuspension} for a
one-fold desuspension in an unpointed setting. 

\subsection{Other posets}
In \Cref{powsetdeloop}, we showed that if $\mathcal{P}(S)$ is the poset of
nonempty subsets of a given finite set $S$ and $\mathcal{P}_{\leq r}(S) \subset
\mathcal{P}(S)$ is the poset of nonempty subsets of cardinality $\leq r$, then
the inclusion $\mathcal{P}_{\leq r}(S) \subset \mathcal{P}(S)$ is $(2r - |S| +
1)$-fold deloopable. We note here that other similar examples of deloopable poset
inclusions occur ``in nature.'' The study of homotopy types of posets and
homotopy colimits of diagrams of spaces indexed over posets has numerous
applications: see for instance \cite{posetfiber, quillenp}. 

Let $k$ be a field, and let $V$ be an $n$-dimensional vector space over $k$.
Consider the poset $\sub(V)$
of nonzero subspaces of $V$. 
Given $2 \leq r \leq n$, let $\sub_{\leq r}(V)$ be the poset of nonzero subspaces of $V$
which have dimension $\leq r$. 

\begin{proposition} 
The inclusion of posets $\sub_{\leq r}(V) \subset \sub(V)$ is $(2r - n +
1)$-fold deloopable.
\end{proposition} 
\begin{proof} 
Let $W \subset V$ be any subspace. We show that the nerve of $\sub_{\leq r}(W)$
(i.e., the left Kan extension of a point from $\sub_{\leq r}(V)$, evaluated on $W$)  has the homotopy
type of a wedge of $(r-1)$-spheres, following an argument of Folkman presented
in \cite[Proposition 8.6]{quillenp}: for each one-dimensional space $L \subset
W$, we
have a subposet 
$\sub_{ \leq r, L}(W)$ consisting of those $\leq r$-dimensional subspaces of $W$ that
contain $L$. The intersection of any $r$ of the subposets $\sub_{\leq r, L}(W)$
contains an initial object (i.e., minimal element) and is thus
weakly contractible. 
Therefore, $\sub_{\leq r}(W)$ is $(r-2)$-connected. 
However, the nerve of $\sub_{\leq r}(W)$ has dimension $\leq r-1$  by
\Cref{dimcat}. 
Putting 
these together, it follows that the nerve of $\sub_{\leq r}(W)$ has the
homotopy type of a wedge of $(r-1)$-spheres. 

Now it follows that $T_{\sub_r(V) \to \sub(V)}$ is pointwise a wedge of
$r$-spheres, and it is weakly contractible when evaluated on any subspace of
dimension $\leq r$. Since the poset $\sub(V) \setminus \sub_r(V)$ has dimension 
$n - r - 1$, the same argument as in \Cref{powsetdeloop} goes through. 
\end{proof}

\subsection{Differentials}

Let $X^\bullet$ be a pointed cosimplicial space, and consider the homotopy
spectral sequence
\[ E_2^{s,t} \implies \pi_{t-s} \mathrm{Tot} X^\bullet,  \]
which comes from the $\mathrm{Tot}$ tower, and the differentials $d_r\colon
E_r^{s,t} \to E_r^{s+r, t+r-1}$.

If $t-s>0$, the entire spectral sequence is a
shift of the 
spectral sequence for $\Omega X^\bullet$. 
All the differentials mapping out of $E_r^{s,t}$ are determined by $\Omega
X^\bullet$ if $t-s>0$. 
But there are also differentials
$d_r\colon E_r^{s,s} \to E_r^{s+r, s+r-1}$ which measure the obstructions to finding
points in $\pi_0 \mathrm{Tot} X^\bullet$. 
These cannot be determined by $\Omega X^\bullet$, as they contribute to $\pi_0$. 
Nonetheless, the $E_2^{s,s}$ terms for $s > 0$ are determined by $\Omega
X^\bullet$. 

\begin{theorem} 
Given any cosimplicial space $X^\bullet$, 
any differential $d_r\colon E_r^{s,s} \to E_r^{s + r, s + r-1}$ depends only on the
cosimplicial space $\Omega X^\bullet$ for 
$r \leq s-1$. 
\end{theorem} 
\begin{proof} 
The differentials $d_k, k \leq r$ starting in filtration $s$ are precisely the
obstructions to lifting an element in $\mathrm{fib}(
\mathrm{Tot}_{s}(X^\bullet) \to
\mathrm{Tot}_{s-1}(X^\bullet))$ to $\mathrm{fib}(
\mathrm{Tot}_{r+s}(X^\bullet) \to
\mathrm{Tot}_{s-1}(X^\bullet))$. 
In particular, if 
\[ 2(s-1) - (r + s ) + 2 \geq 1,  \]
then, by \Cref{strongdeloop}, this obstruction problem is determined entirely by $\Omega X^\bullet$. 
\end{proof}

\begin{corollary} \label{cor:Compare}
Suppose $X^\bullet$ is a pointed cosimplicial space with the properties: 
\begin{enumerate}
\item Any element in $\pi_0 \mathrm{Tot} X^\bullet$ has filtration $\leq N$.  
\item Any class in $E_2^{s,s}$ for $s > N$ is either killed at a finite
stage or supports a differential
$d_r$ for $r \leq s-1$. 
\end{enumerate}
Then if $Y^\bullet$ is a path-connected cosimplicial space such that there exists an
equivalence of pointed cosimplicial spaces $\Omega X^\bullet \simeq \Omega
Y^\bullet$, the same conditions above are valid for the totalization of
$Y^\bullet$. 
\end{corollary} 
\begin{proof} 
This follows from the above theorem. Concretely, any element 
of the fiber $\mathrm{fib}\left(\mathrm{Tot}_s(X^\bullet) \to
\mathrm{Tot}_{s-1}(X^\bullet)\right)$ either maps to the connected component
at $\ast$ in 
$\mathrm{fib}\left(\mathrm{Tot}_s(X^\bullet) \to
\mathrm{Tot}_{s-t-1}(X^\bullet)\right)$ for some $t>0$, which is detected by a
differential that we see from the spectral sequence of $\Omega X^\bullet$
(i.e., not a ``fringed'' differential in the spectral sequence for $X^\bullet$), or it supports a
``fringed'' differential and fails to lift to $\mathrm{fib}(
\mathrm{Tot}_{s+r}(X^\bullet) \to \mathrm{Tot}_{s-1}(X^\bullet))$, where $r
\leq s-1$. But in this
range, everything only depends only on $\Omega X^\bullet$ by
\Cref{strongdeloop}. 
\end{proof}

A natural example of the type of situation in \Cref{cor:Compare} arises from a
descendable map of $\e{\infty}$-ring spectra $S\to R$. 
Denote by $R^\bullet$ the cosimplicial $\e{\infty}$-ring spectrum which is the cobar construction
on the map $S\to R$, also known as the Amitsur complex. Descendability
implies that the spectrum $S$ is equivalent to the totalization of $R^\bullet$. Under mild conditions, the spaces $B(\Omega^\infty R^\bullet)$ and $BGL_1(R^\bullet)$ satisfy the hypotheses of \Cref{cor:Compare}, since for an $\e{1}$-ring spectrum $R$, the space $GL_1(R) \simeq \Omega BGL_1(R)$ is a union of connected components of $\Omega B( \Omega^\infty R)$. We discuss this example further in \cite{MS}: it was the motivating example and application for us.

\bibliographystyle{amsalpha}
\bibliography{biblio}

\end{document}